\documentclass[leqno,11pt]{scrartcl}
\usepackage[utf8]{inputenc}

\usepackage{amsmath,amssymb,amsthm}
\usepackage{amscd}
\usepackage{setspace} 
\usepackage{enumerate}
\usepackage{array}
\usepackage{tabularx}
\usepackage{multirow}
\usepackage{booktabs}

\theoremstyle{definition}
\newtheorem{definition}{Definition}
\newtheorem{remark}{Remark}

\theoremstyle{plain}
\newtheorem{theorem}{Theorem}

\newtheorem{lemma}[definition]{Lemma}
\newtheorem{corollary}{Corollary}

\theoremstyle{remark}

\newcommand{\C}{\mathbb{C}}

\renewcommand{\H}{\mathbb{H}}
\newcommand{\K}{\mathbb{K}}
\newcommand{\N}{\mathbb{N}}
\newcommand{\Pb}{\mathbb{P}}
\newcommand{\Q}{\mathbb{Q}}
\newcommand{\R}{\mathbb{R}}
\newcommand{\Sbb}{\mathbb{S}}
\newcommand{\Z}{\mathbb{Z}}
\newcommand{\Acal}{\mathcal{A}}

\newcommand{\Jcal}{\mathcal{J}}
\newcommand{\Mcal}{\mathcal{M}}

\newcommand{\Scal}{\mathcal{S}}
\newcommand{\Tcal}{\mathcal{T}}

\newcommand{\trace}{\operatorname{trace}\,}
\newcommand{\diag}{\operatorname{diag}\,}

\newcommand{\sgn}{\operatorname{sgn}\,}

\newcommand{\oh}{{\scriptstyle{{\cal O}}}}

\let\leq\leqslant
\let\geq\geqslant

\begin{document}

\begin{center}
\begin{huge}
\begin{spacing}{1.0}
\textbf{On Hermitian Eisenstein Series of Degree $2$}  
\end{spacing}
\end{huge}

\bigskip
by
\bigskip

\begin{large}
\textbf{Adrian Hauffe-Waschbüsch\footnote{Adrian Hauffe-Waschbüsch, Lehrstuhl A für Mathematik, RWTH Aachen University, D-52056 Aachen, adrian.hauffe@rwth-aachen.de}}, 
\textbf{Aloys Krieg\footnote{Aloys Krieg, Lehrstuhl A für Mathematik, RWTH Aachen University, D-52056 Aachen, krieg@rwth-aachen.de}} and 
\textbf{Brandon Williams\footnote{Brandon Williams, Lehrstuhl A für Mathematik, RWTH Aachen University, D-52056 Aachen, brandom.williams@matha.rwth-aachen.de}}
\end{large}
\vspace{0.5cm}\\
May 2022
\vspace{1cm}
\end{center}
\bigskip\textbf{Abstract.} 
We consider the Hermitian Eisenstein series $E^{(\K)}_k$ of degree $2$ and weight $k$ associated with an imaginary-quadratic number field $\K$ and determine the influence of $\K$ on the arithmetic and the growth of its Fourier coefficients. We find that they satisfy the identity $E^{{(\K)}^2}_4 = E^{{(\K)}}_8$, which is well-known for Siegel modular forms of degree $2$, if and only if $\K = \Q (\sqrt{-3})$. As an application, we show that the Eisenstein series $E^{(\K)}_k$, $k=4,6,8,10,12$ are algebraically independent whenever $\K\neq \Q(\sqrt{-3})$. The difference between the Siegel and the restriction of the Hermitian to the Siegel half-space is a cusp form in the Maaß space that does not vanish identically for sufficiently large weight; however, when the weight is fixed, we will see that it tends to $0$ as the discriminant tends to $-\infty$. Finally, we show that these forms generate the space of cusp forms in the Maaß Spezialschar as a module over the Hecke algebra as $\K$ varies over imaginary-quadratic number fields.
\bigskip

\noindent\textbf{Classification MSC 2020:} 11F46, 11F55
\vspace{2ex}\\

\newpage
\section{Introduction}

Eisenstein series are the most common examples of modular forms in several variables. In the case of Hermitian modular forms associated with an imaginary-quadratic number field $\K$, they were introduced by H. Braun \cite{B2}. In this paper we consider Hermitian Eisenstein series of degree $2$. Its Fourier expansion is determined by the Maaß condition and has been worked out explicitly (cf. \cite{K4}, \cite{H-WK20}).

This knowledge leads to new insights on the influence of $\K$ on the arithmetic and the growth of the Fourier coefficients. We will demonstrate that the Eisenstein series $E^{(\K)}_k$ satisfy the identity $E^{{(\K)}^2}_4 = E^{(\K)}_8$, whose analogue for Siegel modular forms of degree $2$ is well known, if and only if $\K=\Q(\sqrt{-3})$. This allows us to show that the  Eisenstein series $E^{(\K)}_k$, $k=4,6,8,10,12,$ are algebraically independent whenever $\K\neq \Q(\sqrt{-3})$.

Finally we consider the difference between $E^{(\K)}_k$ restricted to the Siegel half-space and the Siegel Eisenstein series of weight $k$. This is a Siegel cusp form in the Maaß space. When the weight $k$ is fixed, its limit is $0$ as the discriminant of $\K$ tends to $-\infty$. On the other hand, it does not vanish identically whenever the weight is sufficiently large. Moreover the vanishing of the above difference can be characterized by the vanishing of a Shimura lift. This allows us to show that the subspace of cusp forms in the Maaß space is generated by these restrictions as a module over the Hecke algebra, when $\K$ varies over all imaginary-quadratic number fields.

\section{An identity in weight $8$}

The \emph{Hermitian half-space} $\H_2$ and the \emph{Siegel half-space} $\Sbb_2$ of degree $2$ are given by 
\[\,
 \H_2: = \left\{ Z\in\C^{2\times 2};\;\tfrac{1}{2i} (Z-\overline{Z}^{tr})> 0\right\} \supset \Sbb_2 : = \left\{Z\in\H_2;\; Z=Z^{tr}\right\}.
\]
Throughout the paper we let $\K$ be an imaginary-quadratic number field of discriminant $\Delta = \Delta_\K$ with ring of integers $\oh_\K$ and inverse different $\oh_\K^\sharp = \oh_\K \big/\sqrt{\Delta_\K}$. If $D$ is a fundamental discriminant, let $\chi_D$ denote the associated Dirichlet character; in particular, $\chi_\K = \chi_{\Delta}$. The \emph{Hermitian modular group} of degree $2$ is
\[
 \Gamma^{(\K)}_2:  = \left\{M\in \oh^{4\times 4}_\K;\; \overline{M}^{tr}JM = J,\, \det M = 1\right\}, \; 
 J=\begin{pmatrix}
    0 & -I \\ I & 0
   \end{pmatrix},\;
I = \begin{pmatrix}
    1 & 0 \\ 0 & 1
   \end{pmatrix}.
\]
Then
\[
\Gamma_2:  =\Gamma^{(\K)}_2 \cap \R^{4\times 4}
\]
is the \emph{Siegel modular group} of degree $2$. The space $\Mcal_k(\Gamma^{(\K)}_2)$ of Hermitian modular forms of weight $k$ consists of all holomorphic functions $F:\H_2\to \C$ satisfying
\[
 F\bigl((AZ+B)(CZ+D)^{-1}\bigr) = \det (CZ+D)^k F(Z) \;\;\text{for all}\;\; 
 \left(\begin{smallmatrix}
        A & B\\ C & D
       \end{smallmatrix}\right) \in \Gamma^{(\K)}_2.
\]
Any such $F$ has a Fourier expansion of the form
\[
 F(Z) = \sum_{T\in\Lambda_2, \, T\geq 0} \alpha_F(T) e^{2\pi i \trace(TZ)},
\]
where 
\[
 \Lambda_2 = \left\{ T = \begin{pmatrix}
                          n & t \\ \overline{t} & m
                         \end{pmatrix};\; m,n\in \N_0,\,t\in\oh^\sharp_\K\right\}.
\]
If $\varepsilon(T): = \max\{\ell\in\N;\; \frac{1}{\ell}T \in \Lambda_2\}$ for $T\neq 0$, we can define the Hermitian Eisenstein series of even weight $k\geq 4$ due to \cite{K4} and \cite{H-WK20} as a Maaß lift via
\begin{gather*}\tag{1}\label{gl_1}
 E^{(\K)}_k (Z) = 1 + \sum_{0\neq T\in\Lambda_2,\,T\geq 0} \;\sum_{d|\varepsilon(T)} d^{k-1} \alpha^*_k(|\Delta|\det T/d^2) e^{2\pi i \trace(TZ)}, \;\; Z\in\H_2,
\end{gather*}
where $\alpha^*_k = \alpha^*_{k,\Delta}$ is given by
\begin{gather*}\tag{2}\label{gl_2}
 \alpha^*_k(\ell) =
 \begin{cases}
  0, & \text{if}\;\ell\neq 0,\; a_\Delta (\ell) = 0,  \\
  -2k/B_k, & \text{if}\;\ell = 0,  \\
  r_{k,\Delta} \sum\limits_{t|\ell} \varepsilon^{(\Delta)}_{t,\ell} (\ell/t)^{k-2},
  & \text{if}\;\ell> 0, \;a_\Delta(\ell)\neq 0, \\
 \end{cases}
\end{gather*}
where 
\begin{align*}
 r_{k,\Delta} & = \frac{-4k(k-1)}{B_k B_{k-1,\chi}} > 0,   \\
 \varepsilon^{(\Delta)}_{t,\ell} & = \frac{1}{a_\Delta (\ell)} \sum\limits_{\substack{D_1D_2=\Delta\\D_j\,\text{fund. discr.}}} \chi_{D_1}(t)\chi_{D_2}(-\ell/t), \tag{3} \label{gl_3}\\
 a_\Delta(\ell)& = \sharp\{u:\oh^\sharp_\K/\oh_\K;\;\Delta u\overline{u} \equiv \ell \bmod \Delta\} = \prod^r_{j=1}(1+\chi_j(-\ell)),
\end{align*}
if $ \Delta=\Delta_1 \cdot \ldots\cdot \Delta_r$ is the decomposition into prime discriminants and $\chi_j = \chi_{\Delta_j}$.
If $\ell \in\N$ and $a_\Delta (\ell)> 0$, then any $t\mid \ell$ satisfies 
\begin{gather*}\tag{4}\label{gl_4}
\begin{split}
\varepsilon^{(\Delta)}_{t,\ell} & = \prod^r_{j=1} \frac{\chi_j(t) + \chi_j(-\ell/t)}{1+\chi_j(-\ell)}  \\
& = \prod_{j:\gcd(t,\Delta_j)=1} \chi_j(t) \cdot \prod_{j:\gcd(t,\Delta_j)>1} \chi_j(-\ell/t).
\end{split}
\end{gather*}

If $k>4$ is even, we have the absolutely convergent series
\[
 E^{(\K)}_k (Z) = \sum_{\left(\begin{smallmatrix}
                         A & B \\ C & D
                        \end{smallmatrix}\right):\left(\begin{smallmatrix}
					\ast & \ast \\ 0 & \ast
				      \end{smallmatrix}\right)\big\backslash \Gamma^{(\K)}_2}
\det(CZ+D)^{-k}, \;\; Z\in \H_2.
\]
We derive a first result on the growth and the arithmetic of the Fourier coefficients depending on $\K$.

\begin{theorem}\label{theorem_1neu} 
 Let $\K$ be an imaginary-quadratic number field and let $k\geq 4$ be even. Then 
 \begin{flalign*}
 & \text{a)} \qquad  \varepsilon^{(\Delta)}_{t,\ell} = \chi_{D'_t}(t) \chi_{D_t}(-\ell/t) &
 \end{flalign*}
 holds for all $t|\ell$, $\ell\in \N$, $a_\Delta(\ell)>0$, where $D_t$, $D'_t$ are fundamental discriminants satisfying $D_t D'_t = \Delta$, $|D_t|= \gcd(t^\infty,\Delta)$.
 \begin{flalign*}
 & \text{b)} \qquad 0 < r_{k,\Delta} (2-\zeta(k-2)) \ell^{k-2} \leq \alpha^*_{k,\Delta}(\ell) \leq r_{k,\Delta} \zeta(k-2) \ell^{k-2}  &
 \end{flalign*}
 holds for all $\ell\in\N$ with $a_\Delta(\ell) > 0$.  \\
 c) One has 
 \begin{align*}
 \begin{split}
  0  & <  \frac{(2\pi)^{2k-1}}{\zeta(k-1)\zeta(k)(k-2)!(k-1)!} \cdot \frac{1}{|\Delta|^{k-3/2}} \\
  & \leq r_{k,\Delta} \leq  \frac{(2\pi)^{2k-1}}{(2-\zeta(k-1))\zeta(k)(k-2)!(k-1)!}\cdot \frac{1}{|\Delta|^{k-3/2}}.
   \end{split} 
 \end{align*}
 d) If $\ell_1,\ell_2\in\N$ are coprime with $a_\Delta(\ell_j) >0$ and $\gcd(\ell_1\ell_2,\Delta)=1$, then
 \[
  \alpha^*_{k,\Delta}(\ell_1)\cdot \alpha^*_{k,\Delta}(\ell_2) = r_{k,\Delta} \cdot\alpha^*_{k,\Delta}(\ell_1\ell_2).
 \]
\end{theorem}

\begin{proof}
We observe that
\[
  \sgn(B_kB_{k-1,\chi}) = \chi_\K(-1)=-1,
 \]
\begin{gather*}\tag{5}\label{gl_5}
 |B_k| = \frac{2k!\zeta(k)}{(2\pi)^k},  
\end{gather*}\,
\begin{gather*}\tag{6}\label{gl_6}
 \frac{2(k-1)! |\Delta|^{k-3/2}}{(2\pi)^{k-1}} (2-\zeta(k-1)) \leq |B_{k-1,\chi}| \leq \frac{2(k-1)! |\Delta|^{k-3/2}}{(2\pi)^{k-1}} \zeta(k-1),
\end{gather*} 
 \[
   \varepsilon^{(\Delta)}_{t,\ell} = \chi_\K(t), \;\;\text{if}\;\; \gcd(t,\Delta)=1.
 \]
Then the claim follows from \eqref{gl_2}, \eqref{gl_3} and \eqref{gl_4}.
\end{proof}

Inserting estimates for the Riemann zeta function we get
\begin{alignat*}{2}
 \frac{8\,792}{\sqrt{|\Delta|}} & \leq \alpha^*_4(|\Delta|) & & \leq \frac{61\,362}{\sqrt{|\Delta|}},  \\[1ex]
  \frac{181\,995}{\sqrt{|\Delta|}} & \leq \alpha^*_6(|\Delta|) & & \leq \frac{231\,109}{\sqrt{|\Delta|}},  \\[1ex]
  \frac{251\,164}{\sqrt{|\Delta|}} & \leq \alpha^*_8(|\Delta|) & & \leq \frac{264\,410}{\sqrt{|\Delta|}},  \tag{7}\label{gl_7}\\[1ex]
  \frac{99\,324}{\sqrt{|\Delta|}}  & \leq \alpha^*_{10}(|\Delta|) & & \leq \frac{100\,541}{\sqrt{|\Delta|}}, \\[1ex]
  \frac{15\,720}{\sqrt{|\Delta|}} & \leq \alpha^*_{12}(|\Delta|) & & \leq \frac{15\,768}{\sqrt{|\Delta|}}. 
\end{alignat*}

\begin{corollary}\label{corollary_1neu} 
Let $\K$ be an imaginary-quadratic number field. Then 
 \[
 E^{{(\K)}^2}_4 = E^{(\K)}_8 
\] 
holds if and only if $\K=\Q(\sqrt{-3})$.
\end{corollary}

\begin{proof}
 If $\K= \Q(\sqrt{-3})$, then the identity follows from \cite{DeKr03}, Theorem 6. Suppose $|\Delta|\geq 4$. The Fourier coefficient of $I$ in $E^{{(\K)}^2}_4 - E^{(\K)}_8$ is
\[
   2 \alpha^*_4 (|\Delta|) + 2\alpha^*_4 (0)^2 -\alpha^*_8(|\Delta|) 
   \geq \frac{17\,584}{\sqrt{|\Delta|}} + 115\,200 - \frac{264\,410}{\sqrt{|\Delta|}} > 0,
\]
according to \eqref{gl_7}, whenever $|\Delta|\geq 5$. For $\Delta = -4$ a direct computation shows that this Fourier coefficient is nonzero.
\end{proof}

Clearly the restriction of $E^{{(\K)}^2}_4 - E^{(\K)}_8$ to $\Sbb_2$ vanishes because $\dim \Mcal_8(\Gamma_2)=1$. If $\K = \Q(\sqrt{-1})$ then \cite{DeKr03}, Theorem 10, yields
\begin{gather*}\tag{8}\label{gl_8}
 E^{{(\K)}^2}_4 - E^{(\K)}_8 = c\,\phi^2_4 \;\;\text{for}\;\; c = 230\,400/61.
\end{gather*}

\begin{corollary}\label{corollary_1} 
 Let $\K = \Q(\sqrt{-1})$. Then the graded ring of symmetric Hermitian modular forms with respect to the maximal discrete extension of $\Gamma^{(\K)}_2$ is the polynomial ring in 
 \[
  E^{(\K)}_k, \;\; k= 4,6,8,10,12.
 \]
\end{corollary}

\begin{proof}
 \cite{DeKr03}, Corollary 9 and \eqref{gl_8}.
\end{proof}

Let $e^{(\K)}_{k,m}: \H_1 \times \C \times \C \to \C$ denote the $m$-th Fourier-Jacobi coefficient of $E_k^{(\K)}$ belonging to $J_{k,m}(\oh_\K)$, the space of Hermitian Jacobi forms of weight $k$ and index $m$ (cf. \cite{Hav96}). Note that the first Fourier-Jacobi coefficient of $E^{{(\K)}^2}_4 - E^{(\K)}_8$vanishes on the submanifold $\{(\tau,z,z);\;\tau\in\H_1,\,z\in\C\}$. Let $\Mcal_k(\Gamma_1)$ stand for the space of elliptic modular forms of weight $k$. Then the result of Eichler-Zagier \cite{EZ}, Theorem 3.5, yields

\begin{corollary}\label{corollary_2} 
 Let $\K$ be an imaginary-quadratic number field, $\K\neq \Q(\sqrt{-3})$. If $k\geq 4$ is even, the mapping
 \begin{align*}
  \Mcal_{k-4}(\Gamma_1) \times \Mcal_{k-6}(\Gamma_1)\times \Mcal_{k-8}(\Gamma_1) & \to J_{k,1}(\oh_\K),  \\
  (f,g,h) & \mapsto f e^{(\K)}_{4,1} +g\,e^{(\K)}_{6,1} + h\,e^{(\K)}_{8,1} 
 \end{align*}
is an injective homomophism of the vector spaces, which turns out to be an isomorphism for $\K=\Q(\sqrt{-1})$.
\end{corollary}

\begin{proof}
 Note that the dimensions on both sides are equal to $\left[\frac{k}{4}\right]$ due to \cite{DeKr03}, Theorem 3, whenever $\K = \Q(\sqrt{-1})$.
\end{proof}

We give a precise description of $e^{(\K)}_{k,1}$.

\begin{lemma}\label{lemma_1} 
 Let $\K$ be an imaginary-quadratic number field and let $k\geq 4$ be even. Then the first Fourier-Jacobi coefficient of $E^{(\K)}_k$ is given by
 \begin{align*}
  & e^{(\K)}_{k,1}(\tau,z,w) \\
  & = \tfrac{1}{2} \sum_{\substack{c,d\in\Z \\ \gcd(c,d)=1}} \sum_{\lambda\in\oh_\K} (c\tau+d)^{-k} \exp\left(2\pi i\bigl[(a\tau+b)\lambda\overline{\lambda}-czw+(z\lambda+w\overline{\lambda})\bigr]\big/(c\tau+d)\right).
 \end{align*}
\end{lemma}

\begin{proof}
 Proceed in the same way as Eichler/Zagier \cite{EZ} in \S\,6. One knows that $E^{(\K)}_k$ is an eigenform under all Hecke operators $\Tcal_2(p)$ for all inert primes $p$ from \cite{K4}. On the other hand the Jacobi-Eisenstein series is an eigenform under
 \begin{align*}
  \Tcal_J(p) & = \Gamma^{(\K)}_J \diag(1,p,p^2,p) \Gamma^{(\K)}_J,   \\
  \Gamma^{(\K)}_J & = \left\{\left(\begin{smallmatrix}
                                 * & * & * & * \\ 0 & 0 & 0 & 1
                                \end{smallmatrix}\right) \in \Gamma^{(\K)}_2\right\}, \;\; p\in\Pb\;\;\text{inert}.
 \end{align*}
As in both cases the constant Fourier coefficient is $1$, the claim follows.
\end{proof}

\begin{remark}\label{remark_1} 
 a) If $E_k$ denotes the normalized Eisenstein series of weight $k$ for some group $\Gamma$, then the identity $E^2_4 = E_8$ is well-known for elliptic modular forms and Siegel modular forms of degree $2$ (cf. \cite{Mi89}, \cite{F}). But it also holds for modular forms of degree $2$ with respect to the Hurwitz order (cf. \cite{K3}, p. 89) as well as the integral Cayley numbers (cf. \cite{DKW}), i.e. for $O(2,6)$ and $O(2,10)$. Hence this identity is a hint at the influence of the arithmetic of the attached number field on the modular forms.  \\
 b) Note that $E^{(\K)}_k$ is a modular form with respect to the maximal discrete extension of $\Gamma^{(\K)}_2$ (cf. \cite{KRaW}, \cite{We2}). \\
 c) It follows from \eqref{gl_4} that the Fourier coefficients $\varepsilon^{(\Delta)}_{t,\ell}$ are also multiplicative in $\Delta$, i.e.
 \[
  \varepsilon^{(\Delta)}_{t,\ell} = \varepsilon^{(\Delta_1)}_{t,\ell}\cdot \ldots\cdot \varepsilon^{(\Delta_r)}_{t,\ell}.
 \]
 d) Due to Corollary \ref{corollary_2} the dimension of the Maaß space in $\Mcal_k(\Gamma^{(\K)}_2)$ is $\geq \left[\frac{k}{4}\right]$ for $\K\neq \Q(\sqrt{-3})$ and equal to $\left[\frac{k+2}{6}\right]$ for $\K=\Q(\sqrt{-3})$ (cf. \cite{DeKr03}), if $k\in \N$ is even.  \\
 e) If $k> 4$ is even we can improve the estimate from \cite{K4} slightly for all $T\in\Lambda_2$, $T>0$:
 \begin{align*}
  & \frac{(2\pi)^{2k-1}}{(k-2)!(k-1)!} \cdot \frac{2-\zeta(k-2)}{\zeta(k-1)\zeta(k)}\cdot \frac{1}{\sqrt{|\Delta|}}(\det T)^{k-2} \leq \alpha_k(T)\\
  & \leq  \frac{(2\pi)^{2k-1}}{(k-2)!(k-1)!} \cdot \frac{\zeta(k-3)\zeta(k-2)}{(2-\zeta(k-1))\zeta(k)}\cdot \frac{1}{\sqrt{|\Delta|}}(\det T)^{k-2}.
 \end{align*}

\end{remark}

\section{Algebraic independence}

It is well-known that there are exactly $5$ algebraically independent Hermitian modular forms. 
In this section we explicitly determine algebraically independent Eisenstein series. 

We define the Siegel Eisenstein series $S_k$ of degree $2$ for even $k\geq 4$
\[
 S_k (Z) = \sum_{\left(\begin{smallmatrix}
                         A & B \\ C & D
                        \end{smallmatrix}\right):\left(\begin{smallmatrix}
					\ast & \ast \\ 0 & \ast
				      \end{smallmatrix}\right)\big\backslash \Gamma_2}
\det(CZ+D)^{-k}, \;\; Z\in \Sbb_2,
\]
and denote its Fourier coefficients by $\gamma_k(R)$. Clearly $E^{(\K)}_k|_{\Sbb_2} = S_k$ holds for $k=4,6,8$. The following Fourier coefficients of $S_k$ were computed by Igusa \cite{Ig62} and are given by \\[2ex]
\renewcommand{\arraystretch}{1.5} %
 \setlength{\tabcolsep}{3mm}
 \begin{center}
\begin{tabular*}{0.58\linewidth}{c|r|r|r}
 $f$ & $\gamma_f\left(\begin{smallmatrix}
                     1 & 0 \\ 0 & 0
                   \end{smallmatrix}\right)$ & 
 $\gamma_f\left(\begin{smallmatrix}
                     1 & 0 \\ 0 & 1
                    \end{smallmatrix}\right)$ & 
 $\gamma_f\left(\begin{smallmatrix}
                     1 & 1/2 \\ 1/2 & 1
                    \end{smallmatrix}\right)$  \\ \hline
 $S_4$ & 240 & 30\,240 & 13\,440  \\
 $S_6$ & -504 & 166\,320 & 44\,352  \\
 $S_4 S_6$ & -264 & -45\,360 & 57\,792  \\ 
 $X_{12}$ & 65\,520 & 402\,585\,120 & 39\,957\,120  \\ 
 \end{tabular*}\\[2ex]
 \end{center}
$X_{12}:=441 S^3_4 + 250 S^2_6$.

\begin{lemma}\label{lemma_2} 
 Let $\K$ be an imaginary quadratic number field. Then
 \[
  F^{(\K)}_{10} : = E^{(\K)}_{10} - E^{(\K)}_{4} E^{(\K)}_{6}, \;\;  F^{(\K)}_{12} : = E^{(\K)}_{12} -\frac{441}{691} E^{{(\K)}^3}_{4} - \frac{250}{691} E^{{(\K)}^2}_{6}
 \]
are Hermitian cusp forms of weight $10$ resp. $12$, whose restrictions to $\Sbb_2$ do not vanish identically.
\end{lemma}

\begin{proof}
 If $F=F^{(\K)}_{10}$, $F^{(\K)}_{12}$, then \eqref{gl_1} - \eqref{gl_3} show that $\alpha_F(T) = 0$ for all $T\in \Lambda_2$, $\det T = 0$. Hence $F$ is a cusp form. The Fourier coefficients $\beta_F(R)$ of $F|_{\Sbb_2}$ are given by 
 \[
  \beta_F(R) = \sum_{\substack{T\in \Lambda_2, \,T\geq 0 \\ T+ \overline{T} = 2R}} \alpha_f(T).
 \]
Note that $E^{(\K)}_k |_{\Sbb_2}= S_k$ for $k=4,6,8$. 
If $k=10$, then Theorem \ref{theorem_1neu} and the above table yield $\beta_F(I)>0$. \\
If $k=12$, then for $\Delta\neq -4$ 
\begin{gather*}\tag{9}\label{gl_9}
\begin{split}
  \beta_F(I) & = \alpha^*_{12}(\Delta) 
  + 2\sum_{1\leq j<\sqrt{|\Delta|}} \alpha^*_{12}(|\Delta|-j^2) -\frac{402\,585\,120}{691}  \\
 & \leq 15\,768 \left(\frac{1}{\sqrt{|\Delta|}}+2\right)-\frac{402\,585\,120}{691} < 0
 \end{split}
\end{gather*}
by means of Theorem \ref{theorem_1neu}. If $\Delta = -4$ then 
\[
 \beta_F(I) = -\frac{20\,026\,621\,440\,000}{34\,910\,011} < 0.
\]
\vspace*{-7ex}\\
\end{proof}
\vspace{2ex}

A simple consequence is 

\begin{theorem}\label{theorem_2} 
 Let $\K$ be an imaginary-quadratic number field.  \\
 a) The graded ring of Siegel modular forms of even weight is generated by 
 \[
  E^{(\K)}_k |_{\Sbb_2}, \; k= 4,6,10,12.
 \]
 b) If $\K \neq \Q(\sqrt{-3})$ the Eisenstein series
 \[
  E^{(\K)}_k, \; k= 4,6,8,10,12
 \]
are algebraically independent.
\end{theorem}

\begin{proof}
 a) Use Lemma \ref{lemma_2}. \\
 b) Apply a) as well as Corollary \ref{corollary_1neu}.
\end{proof}

If $\K=\Q(\sqrt{-3})$, we already know the graded ring of Hermitian modular forms (cf. \cite{DeKr03}, Theorem 6).

\begin{corollary}\label{corollary_3} 
 If $\K = \Q(\sqrt{-3})$ the graded ring of symmetric Hermitian modular forms of even weight with respect to $\Gamma^{(\K)}_2$ is the polynomial ring in 
 \[
  E^{(\K)}_k, \; k=4,6,10,12,18.
 \]
\end{corollary}

\begin{proof}
 Use \cite{DeKr03}, Theorem 6, and show that 
 \[
 E^{(\K)}_{18}, \; E^{(\K)}_{12}\cdot E^{(\K)}_{6}, \; E^{(\K)}_{10}\cdot E^{{(\K)}^2}_{4}, \; E^{{(\K)}^3}_{6}, \; E^{(\K)}_{6}\cdot E^{{(\K)}^3}_{4} 
 \]
are linearly independent by calculating a few Fourier coefficients using \eqref{gl_1} - \eqref{gl_4}.
\end{proof}

\begin{remark}\label{remark_2} 
 a) It follows from the results of \cite{BEF16} that there is a non-trivial cusp form $f_4^{(\K)}$ of weight $4$ for all discriminants except $\Delta_{\K} = -3, -4, -7, -8, -11, -15, -20, -23$. As its restriction to the Siegel half-space vanishes identically, one may replace $E_8^{(\K)}$ by $f_4^{(\K)}$ in these cases in Theorem \ref{theorem_2} b).  \\
 b) Using Theorem \ref{theorem_2} resp. Corollary \ref{corollary_3} resp. part a) we can construct a non-trivial skew-symmetric Hermitian modular form of weight $44$ resp. $54$ resp. $40$ by an application of a suitable differential operator (cf. \cite{Ao02}). 
\end{remark}

\section{A Siegel cusp form}

We consider
\begin{gather*}\tag{10}\label{gl_10}
 \begin{split}
  G^{(\K)}_k(Z) : & = E^{(\K)}_k (Z) - S_k(Z),\;\; Z\in \Sbb_2,  \\
  & = \sum_{\substack{\left(\begin{smallmatrix}
                         A & B \\ C & D
                        \end{smallmatrix}\right):\left(\begin{smallmatrix}
					\ast & \ast \\ 0 & \ast
				      \end{smallmatrix}\right)\big\backslash \Gamma^{(\K)}_2\\
				      C^\sharp D\,\not\in \,\R^{2\times 2}}}
\det(CZ+D)^{-k},
 \end{split}
\end{gather*}
where $\left(\begin{smallmatrix}
                         c_1 & c_2 \\ c_3 & c_4
                        \end{smallmatrix}\right)^\sharp = \left(\begin{smallmatrix}
					c_4 & -c_2 \\ -c_3 & c_1
				      \end{smallmatrix}\right)$.
For $\K = \Q(\sqrt{-1})$ this modular form was introduced by Nagaoka and Nakamura \cite{NN}.

\begin{theorem}\label{theorem_3} 
 Let $\K$ bei an imaginary-quadratic number field. If $k\geq 10$ is even, then $G^{(\K)}_k$ is a Siegel cusp form of degree $2$ and weight $k$ in the Maaß space. \\
 a) One has 
\begin{gather*}\tag{11}\label{gl_11}
 \lim_{|\Delta_\K| \to \infty} G^{(\K)}_k (Z) = 0 \;\; \text{for all}\;\; Z\in \Sbb_2. 
\end{gather*}
 b) $G^{(\K)}_k \not\equiv 0$ holds whenever $k\geq \frac{10}{3}|\Delta_\K|+1$.
\end{theorem}

\begin{proof}
 $G^{(\K)}_k$ is a cusp form, as all its Fourier-coefficients $\beta_k(R)$ with $\det R=0$ vanish due to \eqref{gl_1} - \eqref{gl_4}. It belongs to the Maaß space by virtue of \cite{K4} and \cite{H-WK20}.  \\
 a) Because $\dim \Mcal_k(\Gamma_2)\leq 1$ for $k<10$ (cf. \cite{F}), we have $G^{(\K)}_k \equiv 0$ for $k=4,6,8$. Let $k\geq 10$ and
 \[
  S_1 = \begin{pmatrix}
         0 & 0 & P \\ 0 & -S & 0 \\ P & 0 & 0
        \end{pmatrix}, \; P= \begin{pmatrix}
			    0 & 1 \\ 1 & 0
			     \end{pmatrix}, \; S = \begin{pmatrix}
					           2 & 0 \\ 0 & |\Delta|/2
						  \end{pmatrix}, \;\;\text{resp.}\;\;
	\begin{pmatrix}
	 2 & 1 \\ 1 & (|\Delta|+1)/2
	\end{pmatrix},
 \]
if $\Delta$ is even resp. odd. Then the explicit isomorphism in \cite{KRaW} yields
\begin{gather*}\tag{12}\label{gl_12}
 G^{(\K)}_k (Z) = \sum_{\substack{h\in\Z^6,\,h_4\geq 1\\h^{tr} S_1 h = 0,\,\gcd(S_1 h)=1}}
 \biggl(-h_1\det Z + \trace\left(\begin{pmatrix}
                                  h_2 & g \\ \overline{g} & h_5
                                 \end{pmatrix} \cdot Z\right) +h_6\biggr)^{-k},
\end{gather*}
where $g=h_3 + h_4 \sqrt{\Delta}/2$ resp. $g= h_3 + h_4 (1+\sqrt{\Delta})/2$.
By virtue of \cite{K3}, V.2.5, it suffices to show that the series in \eqref{gl_12} over the absolute values at $Z= iI$ tends to $0$, as $|\Delta|\to \infty$. Hence we consider
\begin{align*}
 I_\Delta &: = \sum_h\big| h_1+h_6+i(h_2+h_5)\big|^{-k}  \\
 & = \sum_h \left(h^2_1 + h^2_6 + h^2_2 + h^2_5 +(h_3\,h_4)S\binom{h_3}{h_4}\right)^{-k/2}
\end{align*}
in view of $h^{tr} S_1 h=0$. As $h_4\geq 1$ we get
\[
 I_\Delta \leq \sum_h \left(\frac{|\Delta|-3}{2} + \frac{1}{2} h^{tr} h\right)^{-k/2} \leq (|\Delta|-3)^{-k/4} \cdot \sum_h(h^{tr} h)^{-k/4}.
\]
If $k\geq 14$ we can use the Epstein zeta function for $I^{(6)}$ in order to obtain $\lim_{|\Delta|\to \infty} I_\Delta = 0$. For $k=10$ and $k=12$ apply
\[
 a+b \geq \sqrt[4]{a} \cdot \sqrt[4]{b^3}\;\;\text{for}\;\; a,b > 0
\]
and proceed in the same way. \\
b) Let $\Delta$ be fixed. Then the Fourier coefficient $\beta_k(I)$ of $G^{(\K)}_k$ is given by
\[
 \beta_k(I) = \alpha^*_k(\Delta) +2\sum_{1\leq j\leq \sqrt{\Delta}} \alpha^*_k (\Delta-j^2)-\gamma_k(I).
\]
Due to Maaß \cite{Maa64}
\[
 0 < \gamma_k(I) = \frac{-4k B_{k-1,\chi_{-4}}}{B_k B_{2k-2}}  
\]
holds. Using Corollary \ref{corollary_1neu} this leads to               
\begin{align*}
 \beta_k(I) & \geq r_{k,\Delta}\left((2-\zeta(k-2))\biggl(|\Delta|^{k-2}+2\sum_{1\leq j\leq \sqrt{|\Delta|}}(|\Delta|-j^2)^{k-2}\biggr) -\frac{B_{k-1,\chi_{-4}} B_{k-1,\chi}}{(k-1)B_{2k-2}}\right) \\[1ex]
 & \geq r_{k,\Delta} |\Delta|^{k-3/2}\left(\frac{2-\zeta(k-2)}{\sqrt{|\Delta|}} - \frac{2^{2k-2}(k-1)!^2 \zeta(k-1)^2}{(k-1)\zeta(2k-2)(2k-2)!}\right)
\end{align*}
for $\Delta \neq -4$, if we use \eqref{gl_5} and \eqref{gl_6}. Then Stirling's formula leads to 
\begin{align*}
 \beta_k(I) & \geq r_{k,\Delta}|\Delta|^{k-3/2} \left(\frac{2-\zeta(k-2)}{\sqrt{|\Delta|}} - \frac{e^{1/6(k-1)} \zeta(k-1)^2}{\zeta(2k-2)} \sqrt{\frac{\pi}{k-1}}\right)  \\[1ex]
 & \geq r_{k,\Delta} |\Delta|^{k-3/2} \left(\frac{2-\zeta(8)}{\sqrt{|\Delta|}} - e^{1/54} \zeta(9)^2 \sqrt{\frac{\pi}{k-1}}\right),
\end{align*}
 as $k\geq 10$. The expression in the bracket is positive because
 \[
  k-1 \geq \tfrac{10}{3} |\Delta| > \pi \left(\frac{e^{1/54} \zeta(9)^2}{2-\zeta(8)}\right)^2 |\Delta|.
 \]

If $\Delta=-4$ we compute $\beta_k(R)$,  $R=\left(\begin{smallmatrix}
                                                  1 & 1/2  \\ 1/2 & 1
                                                 \end{smallmatrix}\right)$, 
and proceed in the same way (cf. \cite{NN}). 
\end{proof}

Use the description of $\gamma_k(R)$ by means of Cohen's function in \cite{EZ}, p.\,80. A comparison of the Fourier coefficients and the Hecke bound for the Fourier coefficients of cusp forms yield the following asymptotic.

\begin{corollary}\label{corollary_4} 
 If $k\geq 4$ is even and $N\in \N$, $N\equiv 0,3 \bmod 4$ one has 
 \[
  H(k-1,N) \sim \sum_{\substack{|j|\leq \sqrt{|\Delta| N} \\ j\equiv \Delta N\bmod 2}} \alpha^*_{k,\Delta} \bigl((|\Delta| N-j^2)/4\bigr)
 \]
as $N\to \infty$ for any imaginary-quadratic number field $\K$. 
\end{corollary}

\begin{remark}\label{remark_3} 
a) We know that $G^{(\K)}_k\equiv 0$ for $k=4,6,8$. Hence we get equality for $k=4,6,8$ in Corollary \ref{corollary_4}. 
We conjecture that $G^{(\K)}_k \not\equiv 0$ for any even $k\geq 10$ and any imaginary-quadratic number field $\K$. This has been verified for $|\Delta_\K|\leq 500$ by the authors. The Fourier coefficients are not always positive as in the proof of Theorem \ref{theorem_3} (cf. \cite{H-W22}). \\
b) The paramodular group of level $t$ can be embedded into $\Gamma^{(\K)}_2$, whenever $t$ is the norm of an element in $\oh_\K$ (cf. \cite{Koe2}). Hence one can construct paramodular cusp forms in the Maaß space in the same way.
\end{remark}

\section{The Maaß Spezialschar}

At first we characterize the vanishing of $G^{(\K)}_k$. Recall (e.g. \cite{On04}, Theorem 3.14) that for even $k\in\N$ the \emph{Shimura lifts} are maps
\begin{align*}
 Sh_{k,t-1/2} : &\; \Scal_{k-1/2}(\Gamma_0(4)) \to \Scal_{2k-2}(SL_2(\Z))  \\
 & \; f(\tau) = \sum^\infty_{n=1} a(n) e^{2\pi in\tau} \mapsto \sum^\infty_{n=1} b_t(n)e^{2\pi in\tau},
\end{align*}
for squarefree $t\in\N$, where the coefficients $b_t(n)$ are given by
\[
 \sum^{\infty}_{n=1}b_t(n) n^{-s} = \sum^{\infty}_{n=1}\left(\frac{t}{n}\right) n^{k-s-3/2} \cdot\sum^{\infty}_{n=1} a(tn^2) n^{-s}.
\]
Let $\Scal^+_{k-N_2}(\Gamma_0(4))$ be Kohnen's plus space, i.e.
\[
 a(n) = 0 \;\;\text{for}\;\; n\equiv 1,2 \bmod{4}.
\]
\begin{lemma}\label{lemma_3} 
 Let $k\geq 4$ be even and $\K=\Q(\sqrt{-t})$, $t\in \N$ squarefree. Then the following are equivalent:
 \begin{itemize}
  \item[(i)] The restriction of the Hermitian Eisenstein series $E^{(\K)}_k\big |_{\Sbb_2}$ is equal to the Siegel Eisenstein series $S_k$.
  \item[(ii)] The $t$-Shimura lift $Sh_{t,k-1/2} : \Scal^+_{k-1/2}(\Gamma_0(4)) \to \Scal_{2k-2}(SL_2(\Z))$ is identically zero.
 \end{itemize}
\end{lemma}

\begin{proof}
 As $E^{(\K)}_k\big |_{\Sbb_2}$ is the Maaß lift of $e^{(\K)}_k(\tau,z,z)$, condition (i) holds if and only if $e^{(\K)}_k(\tau,z,z)$ is equal to the classical Jacobi-Eisenstein series in \cite{EZ}. Due to Lemma \ref{lemma_1} one has 
\[                                                                                                                                                                                                                                        e^{(\K)}_k(\tau,z,z) = \sum_{\lambda\in\oh_\K} \sum_{M:\Gamma_\infty\backslash\Gamma_1} e^{2\pi i(\tau\lambda\overline{\lambda}+z(\lambda+\overline{\lambda}))}\big |_{k,1} M,
\]
if we use the definition of the slash operator from \cite{EZ}. For each fixed $\lambda$ the series
\[
 \sum_{M:\Gamma_\infty\backslash\Gamma_1} e^{2\pi i(\tau\lambda\overline{\lambda}+z(\lambda+\overline{\lambda}))}\big |_{k,1} M = P_{k,\lambda\overline{\lambda},\lambda+\overline{\lambda}} (\tau,z)
\]
is the Jacobi-Poincaré series (cf. \cite{Will18}). Its Petersson inner product with a cusp form $\phi$ is, up to a factor depending on $k$ as well as $(\lambda-\overline{\lambda})^{6-4k}$, equal to the Fourier coefficient of the term $e^{2\pi i(\tau\lambda\overline{\lambda}+z(\lambda+\overline{\lambda}))}$. Thus (i) holds if and only if 
\[                                                                                                                                                                                                                                                                                                                                          \sum_{0\neq \lambda\in \oh_\K} (\lambda-\overline{\lambda})^{6-4k} c(\lambda\overline{\lambda}, \lambda+\overline{\lambda}) = 0                                                                                                                                                                                                                                                                                                                                          \]
holds for every Jacobi cusp form
\begin{gather*}\tag{13}\label{gl_13}
 \phi(\tau,z) = \sum_{n,r} c(n,r) e^{2\pi i(n\tau+rz)} \in\Jcal_{k,1}.
\end{gather*}
We want to rephrase this in terms of half-integral weight modular forms for $\Gamma_0(4)$ as in \cite{EZ}, \S\,5. Given \eqref{gl_13} we attach the modular form
\[
 F(\tau) = \sum_{\substack{N\in\N \\ N= 4n-r^2}} c(n,r) e^{2 \pi i N\tau} \in \Mcal_{k-1/2} (\Gamma_0(4)),
\]
satisfying Kohnens's plus condition. This bijection respects the Petersson inner products up to a trivial factor. The half-integral weight modular form attached to $P_{k,\lambda\overline{\lambda}, \lambda+\overline{\lambda}}$ is therefore the projection to the Kohnen plus space of the Poincaré series of weight $k-1/2$ and index $-(\lambda-\overline{\lambda})^2$. As $\lambda$ runs through $\oh_\K$ these values run through $|\Delta|n^2$, $n\in\N_0$. Hence (i) holds if and only if 
\[
 \sum^{\infty}_{n=1} c(|\Delta|n^2)n^{3-2k} = 0
\]
for every cusp form $F(\tau) = \sum_{N\equiv 0,3 \bmod{4}} c(N) e^{2\pi i N\tau} \in S_{k-1/2}(\Gamma_0(4))$. If such an $F$ is a Hecke eigenform, then its Shimura lift $Sh_{t,k-1/2}(F)$ ist either a Hecke eigenform with the same eigenvalues or identically $0$. In the first case the Euler product of its $L$-function implies that 
\[                                                                                                                                                                                                                                                                                                                                        
 \sum^\infty_{n=1} b_t(n) n^{3-2k} \neq 0                                                                                                                                                                                                                                                                                                                                       \]
and thus
\[
 \sum^\infty_{n=1} a(tn^2) n^{3-2k} \neq 0.
\]
Therefore (i) holds if and only if all the Hecke eigenforms in $\Scal^+_{k-1/2}(\Gamma_0(4))$ map to $0$ under $Sh_{t,k-1/2}$.
\end{proof}

This leads to our final result

\begin{corollary}\label{corollary_6} 
 Let $k\geq 10$ be even. Then the modular form
 \[
  G^{(\K)}_k = E^{(\K)}_k\big|_{\Sbb_2} -S_k
 \]
generate the subspace of cusp forms in the Maaß Spezialschar as a module over the Hecke algebra, as $\K$ varies over all imaginary-quadratic number fields.
\end{corollary}

\begin{proof}
 Let $\Acal_{k-1/2}(\Gamma_0(4))\subseteq \Scal^+_{k-1/2} (\Gamma_0(4))$ be the subspace generated by the preimages of $E^{\Q(\sqrt{t})}_k\big|_{\Sbb_2} -S_k$ under the Maaß lift as $t$ runs through all squarefree numbers in $\N$. In the proof of Lemma \ref{lemma_3} it was shown that a Kohnen plus Hecke eigenform $F\in \Scal^+_{k-1/2} (\Gamma_0(4))$ is orthogonal to $\Acal_{k-1/2} (\Gamma_0(4))$ if and only if its $t$-Shimura lifts $Sh_{t,k-1/2}(F)$ are $0$ for all squarefree $t\in \N$. This cannot happen by a theorem of Kohnen \cite{Koh82}, which guarantees that some linear combination of Shimura lifts yields a Hecke-equivariant isomorphism
 \[
   \Scal^+_{k-1/2} (\Gamma_0(4))\overset{\sim}{\longrightarrow} \Scal_{2k-2}(\Gamma_1).
 \]
It follows that $\Acal_{k-1/2}(\Gamma_0(4))$ generates $\Scal^+_{k-1/2}(\Gamma_0(4))$ as a module over the Hecke algebra. Since Maaß lifts respect Hecke operators, we obtain the claim.
\end{proof}

\begin{remark} 
 a) Lemma \ref{lemma_3} is trivial for $k=4,6,8$ because
 \[
  \Scal_k(\Gamma_2) = \Scal_{2k-2}(\Gamma_1) = \{0\}.
 \]
 b) It is an open question whether the modular forms $G^{(\K)}_k$ span the space of cusp forms in the Maaß space of even weight $k\geq 10$, when $\K$ runs over all imaginary-quadratic number fields.
\end{remark}

\nocite{H-W}

\vspace{6ex}


 \bibliography{bibliography_krieg_2022} 

\begin{thebibliography}{10}

\bibitem{Ao02}
{Aoki, H.}
\newblock {The Graded Ring of Hermitian Modular Forms of Degree $2$.}
\newblock {\em Abh. Math. Sem. Univ. Hamburg}, 72:21--34, 2002.

\bibitem{B2}
{Braun, H.}
\newblock {Hermitian modular functions III.}
\newblock {\em Annals Math.}, 53:143--160, 1951.

\bibitem{BEF16}
{Bruinier, J.H., Ehlen, S. and E. Freitag}.
\newblock {Lattices with Many Borcherds Products.}
\newblock {\em Math. Comp.}, 85:1953--1981, 2016.

\bibitem{DeKr03}
{Dern, T. and A. Krieg}.
\newblock {Graded rings of Hermitian modular forms of degree $2$.}
\newblock {\em Manuscr. Math.}, 110:251--272, 2003.

\bibitem{DKW}
{Dieckmann, C., Krieg, A.} and M.~Woitalla.
\newblock {The graded ring of modular forms on the Cayley half-space of degree
  two.}
\newblock {\em Ramanujan J.}, 48:385--398, 2019.

\bibitem{EZ}
{Eichler, M. and D. Zagier}.
\newblock {\em {The Theory of Jacobi Forms}}.
\newblock Birkhäuser, Boston, Basel, Stuttgart, 1985.

\bibitem{F}
{Freitag, E.}
\newblock {\em {Siegelsche Modulfunktionen}}, volume 254 of {\em {Grundl. Math.
  Wiss.}}
\newblock Springer-Verlag, Berlin, 1983.

\bibitem{H-W}
{Hauffe-Waschbüsch, A.}
\newblock {Verschiedene Aspekte von Modulformen in mehreren Variablen.}
\newblock {https://publications.rwth-aachen.de/record/824384/files/824384.pdf}.
\newblock PhD thesis, RWTH Aachen, 2021.

\bibitem{H-W22}
{Hauffe-Waschbüsch, A.}
\newblock {Tables of some Fourier coefficients of Hermitian modular forms of
  degree $2$.}
\newblock Preprint, 2022.
\newblock {\\http://www.matha.rwth-aachen.de/en/forschung/fouriercoeff.html}.

\bibitem{H-WK20}
{Hauffe-Waschbüsch, A. and A. Krieg}.
\newblock {On Hecke Theory for Hermitian Modular Forms}.
\newblock In {\em Modular Forms and Related Topics in Number Theory}, pages
  73--88. Springer, Singapore, 2020.

\bibitem{Hav96}
{Haverkamp, K.}
\newblock {Hermitian Jacobi forms.}
\newblock {\em Results Math.}, 29:78--89, 1996.

\bibitem{Ig62}
{Igusa, J.-I.}
\newblock {On Siegel modular forms of genus two $I$.}
\newblock {\em Am. J. Math.}, 84:175--200, 1962.

\bibitem{Koe2}
{K\"ohler, G.}
\newblock {Modulare Einbettungen Siegelscher Stufengruppen in Hermitesche
  Mo\-dulgruppen.}
\newblock {\em Math. Z.}, 138:71--87, 1974.

\bibitem{Koh82}
{Kohnen, W.}
\newblock {Newforms of half-integral weight.}
\newblock {\em J. Reine Angew. Math.}, 333:32--72, 1982.

\bibitem{K3}
{Krieg, A.}
\newblock {\em {Modular forms on half-spaces of quaternions}}, volume {1143} of
  {\em Lect. Notes Math}.
\newblock \!\!\!, Springer-Verlag, Berlin, 1985.

\bibitem{K4}
{Krieg, A.}
\newblock {The Maaß spaces on the Hermitian half-space of degree $2$}.
\newblock {\em Math. Ann.}, 289:663--681, 1991.

\bibitem{KRaW}
{Krieg, A.}, M.~Raum, and A.~Wernz.
\newblock {The maximal discrete extension of the Hermitian modular group.}
\newblock {\em Documenta Math.}, 26:1871--1888, 2021.

\bibitem{Maa64}
{Maaß, H.}
\newblock {Die Fourierkoeffizienten der Eisensteinreihen zweiten Grades.}
\newblock {\em Mat.-Fys. Medd. Danske Vid. Selsk.}, 34, No. 7, 1964.

\bibitem{Mi89}
{Miyake, T.}
\newblock {\em {Modular forms.}}
\newblock Springer-Verlag, Berlin, 1989.

\bibitem{NN}
{Nagaoka, S. and Y. Nakamura}.
\newblock {On the restriction of the Hermitian Eisenstein series and its
  applications.}
\newblock {\em Proc. Am. Math. Soc.}, 139:1291--1298, 2011.

\bibitem{On04}
{Ono, K.}
\newblock {\em {The web of modularity: arithmetic of the coefficients of
  modular forms and $q$-series}}, volume 120 of {\em {CBMS Regional Conference
  Series in Mathematics.}}
\newblock American Mathematical Society, Providence, RI, 2004.

\bibitem{We2}
{Wernz, A.}
\newblock {Hermitian theta series and Maaß spaces under the action of the
  maximal discrete extension of the Hermitian modular group.}
\newblock {\em Results Math.}, 75:163, 2020.

\bibitem{Will18}
{Williams, B.}
\newblock {Poincaré square series for the Weil representation.}
\newblock {\em Ramanujan J.}, 47:605--650, 2018.

\end{thebibliography}
 \bibliographystyle{plain}

\end{document}